\documentclass[12pt]{amsart}
\usepackage{amsmath,amssymb,amsfonts,amsthm,amsopn}

\newcommand{\modsp}{modulation space}

\newtheorem{tm}{Theorem}[section]
\newtheorem{lemma}[tm]{Lemma}

\newtheorem{theorem}{Theorem}[section]
\newtheorem{corollary}[theorem]{Corollary}

\newtheorem{proposition}[theorem]{Proposition}
\newtheorem{remark}[theorem]{Remark}

\newcommand{\beqa}{\begin{eqnarray*}}
\newcommand{\eeqa}{\end{eqnarray*}}

\newcommand{\field}[1]{\mathbb{#1}}
\newcommand{\bR}{\field{R}}        
\newcommand{\bZ}{\field{Z}}        
\newcommand{\bC}{\field{C}}        
\def\la{\lambda}

 \def\cF{\mathcal{F}}              
 \def\cS{\mathcal{S}}
 \def\cD{\mathcal{D}}
 
 \def\cB{\mathcal{B}}

 \def\cM{\mathcal{M}}

 \def\cC{\mathcal{C}}

 \def\cN{\mathcal{N}}
 
 \def\cT{\mathcal{T}}

\def\a{\aleph}

\def\rd{\bR^d}

\def\zd{\bZ^d}

\def\intrd{\int_{\rd}}

\def\R{\right)}
\def\l{\langle}
\def\r{\rangle}
\def\<{\left<}
\def\>{\right>}

\def\mv1{M_v^1}

\def\o{\xi}
\def\a{\alpha}

\def\ZZ{\mathbb{Z}}

\def\N{\mathbb{N}}
\def\R{\mathbb{R}}
\def\Ren{\mathbb{R}^d}
\def\Renn{\mathbb{R}^{2d}}

\def\Fur{\mathcal{F}}

\def\Sn2{S_{2}(L^{2}(\Ren))}
\def\S1{S_{1}(L^{2}(\Ren))}
\def\sig00{\sigma_{0,0}}

\def\la{\langle}
\def\ra{\rangle}

\begin{document}

\title[]{Remarks on Fourier multipliers
and applications to the Wave
equation}

\author{Elena Cordero and Fabio Nicola}
\address{Department of Mathematics,  University of Torino,
Via Carlo Alberto 10, 10123
Torino, Italy}
\address{Dipartimento di Matematica, Politecnico di
Torino, Corso Duca degli
Abruzzi 24, 10129 Torino,
Italy}
\email{elena.cordero@unito.it}
\email{fabio.nicola@polito.it}



\date{}

\begin{abstract} Exploiting continuity
properties of Fourier
multipliers on modulation
spaces and Wiener amalgam
spaces, we study the
 Cauchy
problem for the NLW equation.
Local wellposedness for rough
data in modulation spaces and
Wiener amalgam spaces is
shown. The results formulated
in the framework of
modulation spaces refine
those in \cite{benyi3}. The
same arguments may apply to
obtain local wellposedness
for the NLKG equation.
\end{abstract}

\maketitle

\section{Introduction and results}
In this short note we study the  Cauchy problem for
 the nonlinear wave equation (NLW):
\begin{equation}\label{cpw}
\begin{cases}
\partial^2_t u-\Delta_x u=F(u)\\
u(0,x)=u_0(x),\,\,
\partial_t u (0,x)=u_1(x),
\end{cases}
\end{equation}
with $t\in\R$, $x\in\R^d$, $d\geq1$, $\Delta_x=\partial^2_{x_1}+\dots \partial^2_{x_d}$. $F$ is a scalar
 function on $\bC$, with
 $F(0)=0$. The solution  $u(t,x)$ is  a complex valued function of $(t,x)\in \R\times\rd$. We will consider
 the case in which $F$ is an entire
 analytic function (in the
 real sense), and we shall highlight the
 special case
 $F(u)=\lambda|u|^{2k}u$,
 $\lambda\in\mathbb{C}$,
 $k\in\mathbb{N}$,
 where we have better
 results.\par
The integral version of the
problem \eqref{cpw} has the
form
\begin{equation}\label{solop}
    u(t,\cdot)=K'(t)u_0+K(t)u_1+\mathcal{B}F(u),
\end{equation}
where\begin{equation}\label{op2}
 K'(t)=\cos(t\sqrt{-\Delta}),\quad
K(t)=\frac{\sin(t\sqrt{-\Delta})} {\sqrt{-\Delta}},\quad
\mathcal{B}=\int _0^t K(t-\tau)\cdot d\tau.
\end{equation}
Here, for every fixed $t$,
the operators $K'(t), K(t)$
in \eqref{op2} are Fourier
multipliers with symbols
$\cos(2\pi t|\o|)$,
$\sin(2\pi t|\o|)/(2\pi
|\o|)$, $\o\in\rd$. We recall
that
 given a function
$\sigma$ on $\rd$ (the
so-called symbol of the
multiplier or,
 simply, multiplier),  the
corresponding Fourier
multiplier operator
$H_{\sigma}$ is formally
defined by
\begin{equation}\label{FM}
H_\sigma f(x)=\intrd e^{2\pi i \xi}\sigma(\xi)
\hat{f}(\xi)\,d\xi.
\end{equation}
So, continuity properties for
multipliers in suitable
spaces yield estimates for
the linear part of the
equation. These latter are
then combined with abstract
iteration (contraction)
arguments to obtain local
wellposedness of \eqref{cpw}.
We deal with this program in
the framework of modulation
spaces and also for Wiener
amalgam
   spaces.
 \par
This was first considered,
for the modulation spaces, in
 \cite{benyi,benyi3,baoxiang2,baoxiang},
 where
  the classical framework of $L^p,$ Sobolev
  or Besov spaces is abandoned
  in favour of such
  spaces,
  which permit to handle initial data
   which are not covered by
   the classical results.
   Precisely, a topic of
   great interest is the
   problem of the
   wellposedness of \eqref{cpw}
   (and other dispersive
   equations) in low
   regularity spaces. The
   classical results in this
   connections (see
   \cite{lindblad-sogge,tao}
   and the references therein)
   use the scale of Soboles
   spaces $H^s$ to measure the local
   regularity of the initial
   data. The modulation
   spaces (and also the Wiener amalgam spaces considered here) provide spaces
   where the local regularity
   is instead measured by the $\Fur
   L^p$-scale (that are the
   spaces of temperate
   distributions whose Fourier
   transform is in $L^p$).
 In order to state our results,
 we first introduce the spaces we
  deal with
  (\cite{F1,grochenig,benyi3}).\par
\label{def1}
 Let $g\in\cS(\rd)$ be a non-zero
window function and consider the so-called
short-time Fourier transform
(STFT) $V_gf$ of a
function/tempered
distribution $f$ with respect
to the the window $g$:
\[
V_g f(x,\o)=\la f, M_{\o}T_xg\ra =\int e^{-2\pi i \o y}f(y)\overline{g(y-x)}\,dy,
\]
i.e.,  the  Fourier transform
$\cF$ applied to
$f\overline{T_xg}$.\par For
$s\in\R$, we consider the
weight function $\la
x\ra^{s}=(1+|x|^2)^{s/2},
x\in\rd. $ If $1\leq p,
q\leq\infty$, $s\in\R$, the
{\it modulation space}
$\mathcal{M}^{p,q}_{s}(\R^n)$
is defined as the closure of
the Schwartz class with
respect to the norm
\[
\|f\|_{\mathcal{M}_{s}^{p,q}}=\left(\intrd\left(\intrd |V_gf(x,\o)|^p dx\right)^{q/p}\la\o\ra^{sq} d\o\right)^{1/q}
\]
(with obvious modifications
when $p=\infty$ or
$q=\infty$).\par Among the
properties of \modsp s, we
record that they are Banach
spaces whose  definition is
independent of the choice of
the window $g\in\cS(\rd)$,
$\mathcal{M}^{2,2}=L^2$,
$\mathcal{M}^{p_1,q_1}_s\hookrightarrow
\mathcal{M}^{p_2,q_2}_s$, if
$p_1\leq p_2$ and $q_1\leq
q_2$,
$(\mathcal{M}_{s}^{p,q})'=\mathcal{M}_{-s}^{p',q'}$.\par
For a more general
definition, involving
different kinds of weight
functions, both in the time
and the frequency variables
we refer to \cite{grochenig}.

The Definition \ref{def1}
has been accordingly extended to the quasi-Banach case  $0<p,q<1$ in \cite{koba1,koba2}. Here the window function $g$ is to be restricted to the class
$$\left\{g\in\cS(\rd)\,:\,\mbox{supp} \,\hat{g}\subset\{\o\,:\,|\o|\leq 1\},\,\,\mbox{ and}\,\,\sum_{k\in\zd}\hat{g}(\o-\a k)\equiv 1,\quad\forall\o\in\rd\right\},
$$
for a sufficiently  small $\a>0$, so that the set above is not empty.
For other definitions of modulation spaces for all
 $0<p,q\leq\infty$, we refer to \cite{galperin-samarah,rauhut05,triebel83}.

We also recall the definition
of the so-called Wiener
amalgam spaces
(\cite{benyi3,feichtinger90,rauhut05}).
First of all we denote by
$L^p_s$, $s\in\R$,
$0<p\leq\infty$, the weighted
$L^p$ space of function $f$
in $\R^d$ such that $\l x\r^s
f(x)$ is in $L^p$, with the
obvious quasi-norm (norm if
$p\geq1$). Then,
 for $s,\gamma\in\R$, $1\leq p,q<\infty$,
  a tempered distribution $f$ is in the Wiener
  amalgam space $W(\cF L^q_s, L^p_\gamma)(\rd)$ if $f$ is locally in $\cF L^q_s(\rd)$, that is, for every non-null $g\in\cC^\infty_0(\rd)$,
$\cF(fT_xg)\in  L^q_s(\rd)$
and
\begin{equation}
\label{asterisco}
\|f\|_{W(\cF
L^q_s,
L^p_\gamma)}=\left(\intrd
\left(\intrd |\cF
(fT_xg)(y)|^q\,\la
y\ra^{sq}dy\right)^{p/q}\la
x\ra^{\gamma
p}dx\right)^p<\infty.
\end{equation}
When $p=\infty$ or
$q=\infty$, we define $W(\cF
L^q_s, L^p_\gamma)(\rd)$ as
the closure of the Schwartz
space with the norm in
\eqref{asterisco} (modified
in the obvious way). This
definition is independent of
the test function
$g\in\cS(\rd)$. For
properties we refer to
\cite{feichtinger90}. The
definition above is extended
to the cases $0<p,q<1$  in
\cite{benyi3}; here the
window function $g\in\cS$
satisfies ${\rm
supp}\,\hat{g}\subset
\{|\o|\leq 1\}$, see also
\cite{rauhut05} for the
quasi-Banach case in the
global component $L^q$.\par
We now briefly present our
wellposedness results (see
the statements of Theorems
\ref{T1}, \ref{T2} and
\ref{T3} for details).
Consider first a nonlinearity
$F(u)$, where $F$ is an
entire real-analytic
function. Then we will prove
that \eqref{cpw} is wellposed
in $\cM^{p,1}_s(\R^d)$ for
every $1\leq p\leq\infty$,
$s\geq0$. In particular, for
every $(u_0,u_1)\in
\cM^{p,1}_s(\R^d)\times
\cM^{p,1}_{s-1}(\R^d)$ there
exists $T>0$ and a unique
solution $u\in
\cC([0,T];\cM^{p,1}_s(\R^d))$
to \eqref{solop}. Similarly,
we shall  show that
\eqref{cpw} is also wellposed
in $W(\Fur L^1_s,L^p_\gamma)$
for every $1\leq
p\leq\infty$,
$s,\gamma\geq0$.\\ We then
consider the case of the
power nonlinearity
$F(u)=\lambda|u|^{2k}u$. We
show that \eqref{cpw} is
wellposed in
$\cM^{p,q}_s(\R^d)$ for every
$1\leq p\leq\infty$,
$q'>2kd$, $s\geq0$. This
refines a result in
\cite{benyi3}, where the
authors assume $q'=\infty$
(and $u_1$ with the same
regularity $s$ as $u_0$,
rather than $s-1$). For the
same nonlinearity we also
prove wellposedness in
$W(\Fur L^q_s,L^p_\gamma)$,
$1\leq p\leq\infty$,
$q'>2kd$, $s,\gamma\geq0$.\\
The most interesting cases
are of course when
$s=\gamma=0$, $p=\infty$, and
$q$ is as large as possible.
In this connection we notice
that $\cM^{\infty,q}_s\subset
W(\Fur L^q_s, L^\infty)$, so
Wiener amalgam spaces allow
us to consider more general
initial data than those in
the above modulation spaces.

Let us highlight that the
same arguments apply to the
study of the Cauchy problem
for the nonlinear
Klein-Gordon equation (NLKG)
(that is \eqref{cpw}  with
the operator $-\Delta_x$
replaced by  $I-\Delta_x$).
We omit details.

The basic tool of the proof
is given by Fourier
multiplier estimates on
\modsp s.
 Precisely, we will use (the first part
 of) the
 following refinement
 of results in
 \cite{benyi,benyi3},
  involving multipliers with
   symbols in the Wiener amalgam spaces
   $W(\cF L^p_s, L^q_\gamma)$.
\begin{proposition} \label{L1} Let $s,t\in\R$, $0<q\leq\infty$. Let $\sigma$ be a function on $\rd$ and consider the
Fourier multiplier operator defined in \eqref{FM}.\\
(i)  If $1\leq p\leq \infty,$ $\sigma\in W(\cF L^1,
 L^\infty_{\gamma})$, then the operator $H_\sigma$
extends to  a bounded
operator from
$\mathcal{M}^{p,q}_s$
into
$\mathcal{M}^{p,q}_{s+\gamma}$,
with
\begin{equation}\label{fumultest}
    \|H_\sigma f\|_{\mathcal{M}^{p,q}_{s+\gamma}}
    \lesssim \|\sigma\|_{W(\cF L^1,
    L^\infty_{\gamma})}\|f\|_{\mathcal{M}^{p,q}_{s}}.
\end{equation}
(ii) If $0< p<1$, $\sigma\in
W(\cF L^p,
L^\infty_{\gamma})$, then the
operator $H_\sigma$ is a
bounded operator from
$M^{p,q}_s$ into
$\mathcal{M}^{p,q}_{s+\gamma}$,
with
\begin{equation}\label{fumultest2}
    \|H_\sigma f\|_{\mathcal{M}^{p,q}_{s+\gamma}}
    \lesssim \|\sigma\|_{W(\cF L^p,
    L^\infty_{\gamma})}\|f\|_{\mathcal{M}^{p,q}_{s}}.
\end{equation}
\end{proposition}
Similar estimates are proved
for multipliers acting on
Wiener amalgam spaces as
well.

\section{Preliminary results} In
this section we collected
some preliminary results.
\par
This first lemma lemma is
proved in \cite[Corollary
4.2]{baoxiang}. The case
$p,p_i\geq 1$ was first
proved by Feichtinger
\cite{F1}. It is also proved
in \cite{benyi3} using the
theory of multi-linear
pseudodifferential operators.
\begin{lemma} \label{L2} Let $s\geq 0$,
 $0<p\leq p_i\leq\infty$, $1\leq r,q_i\leq\infty$, $N\in\N$, satisfy
\begin{equation}\label{indices}
\sum_{i=1}^N \frac
1{p_i}=\frac1p,\quad
\sum_{i=1}^N \frac
1{q_i}=N-1+\frac1r,
                \end{equation}
                then we have
$$
\| \prod_{i=1}^N
u_i\|_{M^{p,r}_s}\leq
\prod_{i=1}^N \|
u_i\|_{M^{p_i,q_i}_s}.
$$
                \end{lemma}
 In particular, for
$p_i=N p$, $q_i=q$,
$i=1,\dots N$, we get
\begin{equation}\label{norms}
\|\prod_{i=1}^N
u_i\|_{M^{p,r}_s}\leq
\prod_{i=1}^N\|
u_i\|_{M^{p,q}_s},\quad
\frac{N}{q}=N-1+\frac1r.
\end{equation}
\begin{lemma}
Let $0<p_i,p\leq\infty$,
$1\leq r,q_i\leq\infty$,
$N\in\mathbb{N}$ satisfy
\eqref{indices}, and
$s\geq0$,
$\gamma=\sum_{i=1}^N\gamma_i$,
$\gamma_i\in\R$. We have
\[
\|\prod_{i=1}^N u_i\|_{W(\Fur
L^{r}_s,L^p_\gamma)}\leq\prod_{i=1}^N
\|u_i\|_{W(\Fur
L^{q_i}_s,L^{p_i}_{\gamma_i})}.
\]
\end{lemma}
In particular, for $p_i=N p$,
$q_i=q$,
$\gamma_i=\gamma/N\geq0$,
$i=1,\dots N$, we get
\begin{equation}\label{agt}
\|\prod_{i=1}^N u_i\|_{W(\Fur
L^{r}_s,L^p_\gamma)}\leq\prod_{i=1}^N
\|u_i\|_{W(\Fur
L^{q}_s,L^{p}_{\gamma})},\quad
\frac{N}{q}=N-1+\frac1r.
\end{equation}
\begin{proof}
We choose a window function
$g$ of the type
$g(x)=\prod_{i=1}^n g_i(x)$,
with
$g_i\in\cC^\infty_0(\R^d)\setminus\{0\}$,
$i=1,\ldots,N$. Then $M_{-x}
\hat{g}=M_{-x}\hat{g}_1\ast\ldots
\ast M_{-x}\hat{g}_N$, so
that
\[
\Fur\left(\prod_{i=1}^N u_j
T_x g\right)=(\hat{u}_1\ast
M_{-x}\hat{g}_1)\ast\ldots\ast
(\hat{u}_N\ast
M_{-x}\hat{g}_N).
\]
Since $s\geq0$, a repeated
application of the inequality
$\langle\xi\rangle^s\leq\l
\eta\r^s\l \xi-\eta\r^s$ and
Young inequality then give
\[
\|\prod_{i=1}^N u_i T_x
g\|_{\Fur L^r_s}\leq
\prod_{i=1}^N \|u_i T_x
g\|_{\Fur L^{q_j}_s}.
\]
Now we multiply this
inequality by $\langle
x\rangle^\gamma$ and conclude
by an application of
H\"older's inequality.
\end{proof}
\begin{lemma}
For $0<p,q\leq\infty$,
$s\in\mathbb{R}$ we have
\begin{equation}\label{incl}
W(\Fur
L^r_{s},L^p_\gamma)\hookrightarrow
W(\Fur
L^q_{s-1},L^p_\gamma),\quad
\frac{d}{q}-\frac{d}{r}<1.
\end{equation}
\end{lemma}
\begin{proof}
An application of H\"older's
inequality shows that
$L^r_{s}\hookrightarrow
L^q_{s-1}$. Hence the desired
embedding follows directly
from the definition of Wiener
amalgam spaces.\\ $ $
\end{proof}

In the sequel the following convolution relations will be useful \cite{F1}.
\begin{lemma}
  For $i=1,2,3$, let $B_i$ be one of the Banach spaces
  $\cF L^q_s$ ($1\leq q\leq\infty$,
   $s\in\R$),  $C_i$ be one of the Banach spaces
  $L^p_\gamma$ ($1\leq p\leq\infty$, $\gamma\in\R$).  If $B_1\ast B_2\hookrightarrow B_3$ and $C_1\ast
  C_2\hookrightarrow C_3$, we have
  \begin{equation}\label{conv0}
  W(B_1,C_1)\ast W(B_2,C_2)\hookrightarrow W(B_3,C_3).
  \end{equation}
\end{lemma}

\section{Multiplier estimates}

In this section we first
prove estimates in
$M^{p,q}_s$ for the
multipliers arising in the
wave propagator, refining
those in \cite{benyi,benyi3}.
Then we prove estimates in
the Wiener amalgam spaces
$W(\Fur
L^p_s,L^q_\gamma)$. We begin with the proof of
Proposition \ref{L1}.

\begin{proof}[Proof of Proposition \ref{L1}] For the case $1\leq p\leq \infty$, the arguments are a rearrangement of
\cite[Lemma 8]{benyi},
whereas the case $0<p<1$ one
argues as in \cite[Lemma
1]{benyi2}. For the sake of
clarity, we shall detail the
first case. Choose a window
function $g_1=g_0\ast
g_0\in\cS(\rd)$ and use $M_\o
g_1=M_\o g_0\ast M_\o g_0$,
so that, by Young inequality,
\begin{align*}
   \| H_\sigma f\|_{\mathcal{M}^{p,q}_{s+\gamma}}^q
   &= \intrd \|(H_\sigma f)\ast M_{\o} g_1 \|_p^q\la
   \o\ra^{q(s+\gamma)}\,d\o\\
   &\leq \intrd \|\cF^{-1}(\sigma) \ast M_{\o} g_0
   \|_{L^1}^q \| f \ast  M_{\o} g_0\|_{L^p}^q\,\la \o
   \ra^{q(s+\gamma)}\,d\o\\
   &\leq \left(\sup_{\o\in\rd} \|\cF^{-1}(\sigma)
   \ast M_{\o}
   g_0\|_{L^1}^q\la\o\ra^{q\gamma}\right) \intrd
   \| f \ast  M_{\o} g_0\|_{L^p}^q\,\la \o \ra^{qs}
   \,d\o\\
   &=\|\sigma\|_{W(\cF L^1,
    L^\infty_{\gamma})}^q\|f\|_{\mathcal{M}^{p,q}_{s}}^q,
\end{align*}
as desired.
\end{proof}

We need also the following
elementary result. We give
the proof for the sake of
completeness.
\begin{lemma}\label{L3}
Let  $R>0$ and
$f\in\cC_0^\infty(\rd)$ such
that supp $f\subset
B(y,R):=\{x\in\rd, |x-y|\leq
R\}$, with $y\in\rd$. Then,
for every $0<p\leq \infty$,
there exist an index
$k=k(p)\in \N$ and a constant
$C_{R,p}>0$ (which depends
only on $R$ and $p$)  such
that
\begin{equation}\label{E1}
                \|f\|_{\cF L^p}\leq C_{R,p}
                \sup_{|\a|\leq k} \|\partial^\alpha f\|_{L^\infty}.
\end{equation}
\end{lemma}
\begin{proof} We know $f\in\cC_0^\infty\subset\cF L^p$.
If we take $k\in\N$ such that
$kp>d/2$, then
\begin{align}
                \|f\|_{\cF L^p}&=\left(\left\|\left(\frac{\la\cdot\ra^{2k}}{\la\cdot\ra^{2k}}|\hat{f}|\right)^p\right\|_{L^1}\right)^{1/p}\\
                &=\left(\intrd \frac1{\la\o\ra^{2kp}}\left|\intrd e^{-2\pi i x \o}(1-\Delta)^k f(x) dx\right|^pd\o\right)^{1/p}\\
                &\leq C_k vol(B(y,R))\sup_{|\a|\leq 2 k}\|\partial^\a f\|_{L^\infty}\leq C_{R,p} \sup_{|\a|\leq 2k}\|\partial^\a f\|_{L^\infty}.
\end{align}
\end{proof}

Now, we consider  the family
of multipliers
$\sigma_{\a,\delta}$, and
$\tau$ defined by
\begin{equation}\label{multsymb}
    \sigma_{\a,\delta}=\frac{\sin |\o|^\a}{|\xi|^\delta},
    \quad
    \delta\leq \a\leq 1,\
    \alpha>0,
\end{equation}
\begin{equation}\label{multsymb2}
\tau(\xi)=\cos|\xi|.
\end{equation}
\begin{proposition}\label{A1} (i) The multipliers
$\sigma_{\a,\delta}$ in
\eqref{multsymb} are in the
space $W(\cF
L^1,L^\infty_\delta)$. (ii)
For $\a=\delta=1$,  we have
$\sigma_{1,1}\in W(\cF
L^p,L^\infty_1)$ for every
$0<p\leq\infty$.\\
(iii) The multiplier $\tau$
in \eqref{multsymb2} is in
$W(\cF L^p,L^\infty)$ for
every $0<p\leq\infty$.
\end{proposition}
\begin{proof}  We first prove ({\it i}) and ({\it ii}). We consider a function $\chi\in\cC^\infty_0(\rd)$, $ 1\leq \chi(\xi)\leq
2$, such that $\chi(\o)=1$ if
$|\o|\leq 1$, whereas
 $\chi(\o)=0$ if $|\o|\geq 2$. Then, we split the multiplier into the sum of two functions $\sigma_{sing}$ and $\sigma_{osc}$, bearing the singularity at the origin and the oscillation at infinity, respectively:
\begin{equation}\label{separazione}
\sigma_{\a,\delta}(\o)=\chi(\o)\sigma_{\a,\delta}(\o)+(1-\chi(\o))\sigma_{\a,\delta}(\o)=\sigma_{sing}(\xi)+\sigma_{osc}(\xi).
\end{equation}
{\em Singularity at the
origin}. First, we shall
prove that   $\sigma_{sing}$
is in $W(\cF
L^1,L^\infty_s)$, for every
$s\in \R$. Indeed, choose
$g\in \cC^\infty_0(\rd)$ such
that supp
$g\subset\{\o\in\rd,
|\o|\leq1\}$, then
$\sigma_{sing}(\o) T_x
g(\o)=0$ if $|x|>3$, so that
\begin{eqnarray*}
  \| \sigma_{sing}\|_{ W(\cF
L^1,L^\infty_s)}&=&
\mbox{ess\,sup}_{x\in\rd}(\|\sigma_{sing}
T_x g \|_{\cF L^1}\la x\ra^s)
=
\mbox{ess\,sup}_{|x|\leq 3}(\|\sigma_{sing} T_x g \|_{\cF L^1}\la x\ra^s)\\
   &\leq&(10)^{|s|/2} \mbox{ess\,sup}_{|x|\leq 3}
   (\|\sigma_{sing} T_x g \|_{\cF
   L^1})\\
   &\leq& (10)^{|s|/2}\|\sigma_{sing}
   \|_{W(\cF L^1,L^\infty)}<\infty,
\end{eqnarray*}
for $\sigma_{sing}\in\cF
L^1\subset W(\cF
L^1,L^\infty)$, (see
\cite[Theorem 9]{benyi}).\par
For $\a=\delta=1$, the
function $\sigma_{sing}$ is
 in $\cC_0^\infty\subset W(\cF L^p,L^\infty)$,
  for every $0<p\leq \infty$.
\par
\noindent {\em Oscillation at
infinity.} At infinity the
multipliers $\sigma_{osc}$
are in $W(\cF
L^p,L^\infty_\delta)$,
 for every $0<p\leq 1$.  The proof uses
  Lemma \ref{L3}, applied to the function
  $\sigma_{osc}T_x g\in\cC^\infty_0(\rd)$, with $g\in\cC^\infty_0(\rd)$.
  Precisely, we observe
  the decay properties of
  $\sigma_{osc}$,
 \begin{equation}\label{aggiunta1}
 |\partial^\a \sigma_{osc}(\o)|\lesssim
  \la\o\ra^{-\delta},\quad \forall \a\in\bZ^d_+
  \end{equation}
  and those of the window $g$,
\begin{equation*}
|\partial^\a g(\o-x)|\lesssim
\la x-\o\ra^{-N},\quad
 \forall x,\o \in\rd,\ N\in\mathbb{N}.
\end{equation*}
Combining the preceding
estimates and the weight
 property
 \begin{equation}\label{aggiunta3}
 \la\o\ra^{-\delta}\la x-\o\ra^{-|\delta|}
 \leq \la x\ra^{-\delta}\end{equation}
  yields
$$\|\sigma_{osc} \|_{W(\cF L^p, L^\infty_\delta)}
=\sup_{x\in\rd}\|\sigma_{osc}T_x g\|_{\cF L^p}\la
 x\ra^\delta<\infty.$$
 Finally, to prove ({\it iii}), observe that
$\tau\in\cC^\infty(\rd)$
fulfills  $$|\partial^\a
\tau(\xi)|\lesssim 1,\quad
\forall \o\in\rd;$$ so,
arguing similarly to what
done for $\sigma_{osc}$
before,  we obtain the claim.
\end{proof}

\noindent
\begin{corollary}\label{C1}
Let $s\in\R$, $0<q\leq \infty$.\\
(i) For every $1\leq
p\leq\infty$, the Fourier
multiplier
$H_{\sigma_{\a,\delta}}$,
with symbol
$\sigma_{\a,\delta}$ defined
in \eqref{multsymb}, extends
to a bounded operator from
$\mathcal{M}^{p,q}_s(\rd)$
into
$\mathcal{M}^{p,q}_{s+\delta}(\rd)$,
with
\begin{equation}\label{fumultest3}
    \|H_{\sigma_{\a,\delta}} f\|_{M^{p,q}_{s+\delta}}\lesssim \|{\sigma_{\a,\delta}}\|_{W(\cF L^1,
    L^\infty_{\delta})}\|f\|_{M^{p,q}_{s}}.
\end{equation}
(ii) For every $0<p<1$,
$H_{\sigma_{1,1}}$ extends to
a bounded operator from
$\mathcal{M}^{p,q}_s(\rd)$
into
$\mathcal{M}^{p,q}_{s+1}(\rd)$
with
\begin{equation}\label{fumultest234}
    \|H_{\sigma_{1,1}} f\|_{M^{p,q}_{s+1}}\lesssim \|{\sigma_{1,1}}\|_{W(\cF L^p,
    L^\infty_{1})}\|f\|_{M^{p,q}_{s}}.
\end{equation}
(iii) For every $0<
p\leq\infty$, the Fourier
multiplier $H_{\tau}$, with
the symbol $\tau$ in
\eqref{multsymb2} extends to
a bounded operator from
$\mathcal{M}^{p,q}_s(\rd)$
into
$\mathcal{M}^{p,q}_{s}(\rd)$,
with
\begin{equation}\label{fumultest4}
    \|H_{\tau} f\|_{M^{p,q}_{s}}
    \lesssim \|{\tau}\|_{W(\cF L^r,
    L^\infty)}\|f\|_{M^{p,q}_{s}},
\end{equation}
where $r=\min\{1,p\}$.
\end{corollary}
\begin{proof}
The desired result follows
from Propositions \ref{A1}
and  \ref{L1}.
\end{proof}

\begin{proposition}\label{prop24}
(i) The functions
$\sigma(\xi)=e^{\pm i|\xi|}$
belong to $M^{\infty,1}$, so
that $\hat{\sigma}\in W(\Fur
L^\infty,L^1)$.\par\noindent
(ii) The function
$\sigma_{1,1}$ in
\eqref{multsymb} satisfies
$\widehat{\sigma_{1,1}}\in
W(\Fur
L^\infty_1,L^1_\gamma)$ for
every $\gamma\in\R$.\\
(iii) The function $\tau$ in
\eqref{multsymb2} satisfies
$\widehat{\tau}\in W(\Fur
L^\infty_1,L^1_\gamma)$ for
every $\gamma\in\R$.
\end{proposition}
\begin{proof}
Part ({\it i}) was proved in
\cite[Theorem 1]{benyi}. In order to
prove the second point, we
split the symbol
$\sigma_{1,1}$ as in
\eqref{separazione} (with
$\alpha=\delta=1$). Then
$\sigma_{sing}\in
\mathcal{C}^\infty_0(\R^d)$,
so that
$\widehat{\sigma}\in\cS(\R^d)\subset
W(\Fur
L^\infty_1,L^1_\gamma)$ for
every $\gamma\in\R$. We now
treat the symbol
$\sigma_{osc}$. We observe
that $\widehat{
\widehat{\sigma_{osc}} T_y
g}=\widetilde{\sigma_{osc}}\ast
M_{-y}\hat{g}$, where
$\widetilde{\sigma_{osc}}(\xi)=\sigma_{osc}(-\xi)$.
Hence we will obtain
$\widehat{\sigma_{osc}} \in
W(\Fur
L^\infty_1,L^1_\gamma)$ if we
prove that
\[
\|\langle\xi\rangle\int
e^{-2\pi i yx} \hat{g}(x)
\sigma_{osc}
(\xi-x)\,dx\|_{L^\infty}\leq
C (1+|y|^2)^{-N}
\]
for an integer $N$ such that
$N-\gamma>d/2$. We multiply
this inequality by
$(1+|y|^2)^{N}$. By an
integration by part and the
Leibniz rule we see that it
suffices to prove that
\[
\langle\xi\rangle\int|\partial^\alpha
\hat{g}(x)\partial^\beta
\sigma_{osc}(\xi-x)|\,dx\leq
C,\quad \forall \xi\in\R^d,\
|\alpha|+|\beta|\leq 2N.
\]
 Then one concludes by
applying \eqref{aggiunta1}
(with $\delta=1$), combined
with the estimate
$|\partial^\a
\hat{g}(x)|\lesssim \langle
x\rangle^{-N'}$, for all
$x\in\rd$, $N'\in\mathbb{N}$,
and \eqref{aggiunta3}.\par
The proof of ({\it iii}) is
completely similar.
\end{proof}
\begin{corollary}\label{C3}
Let $s,\gamma\in\R$, $1\leq
p,q\leq \infty$.\\
(i) The Fourier multiplier
$H_{\sigma_{1,1}}$, with
symbol
$\sigma_{1,1}(\xi)=\frac{\sin|\xi|}{|\xi|}$
extends to a bounded operator
from $W(\Fur
L^q_s,L^p_\gamma)$ into
$W(\Fur
L^q_{s+1},L^p_\gamma)$, with
\begin{equation}\label{fumultest5}
    \|H_{\sigma_{1,1}} f\|_{W(\Fur L^q_{s+1},L^p_\gamma)}
    \lesssim \|{\sigma_{1,1}}\|_{W(\cF L^\infty_1,
    L^1_{|\gamma|})}\|f\|_{W(\Fur L^q_s,L^p_\gamma)}.
\end{equation}
(ii)  The Fourier multiplier
$H_{\tau}$, with symbol
$\tau(\xi)=\cos|\xi|$ extends
to a bounded operator from
$W(\Fur L^q_s,L^p_\gamma)$
into $W(\Fur
L^q_{s},L^p_\gamma)$, with
\begin{equation}
    \|H_{\tau} f\|_{W(\Fur L^q_{s},L^p_\gamma)}
    \lesssim \|{\tau}\|_{W(\cF L^\infty_1,
    L^1_{|\gamma|})}\|f\|_{W(\Fur L^q_s,L^p_\gamma)}.
\end{equation}

\end{corollary}
\begin{proof}
Since $H_{\sigma}
f=\widehat{\sigma}\ast f$,
when $\sigma$ is a temperate
distribution and $f$ is a
Schwartz function, the
desired result follows at
once from Proposition
\ref{prop24} and the
convolution relations
\eqref{conv0} of Wiener
amalgam spaces (recall,
$L^p_\gamma\ast
L^1_{|\gamma|}\hookrightarrow
L^p_\gamma$).
\end{proof}\par

\section{Local  wellposedness of NLW}
In this section we establish
and prove the wellposedness
result outlined in the
Introduction.
\begin{theorem}\label{T1} Assume $s\geq 0$,
 $1\leq p\leq\infty$,
 $(u_0,u_1)\in \mathcal{M}^{p,1}_s(\rd)\times
\mathcal{M}^{p,1}_{s-1}(\rd)$
and $F(z)=\sum_{j,k=0}^\infty
c_{j,k} z^j \bar{z}^k$, an
entire real-analytic function
on $\bC$ with $F(0)=0$. For
every $R>0$, there exists
$T>0$ such that for every
$(u_0,u_1)$ in the ball $B_R$
of center $0$ and radius $R$
in $\mathcal{M}^{p,1}_s(\rd)
\times\mathcal{M}^{p,1}_{s-1}(\rd)$
there exists a unique
solution $u\in
\cC^0([0,T];\mathcal{M}^{p,1}_s(\rd))$
to \eqref{solop}.
Furthermore, the map
$(u_0,u_1)\mapsto u$ from
$B_R$ to $
\cC^0([0,T];\mathcal{M}^{p,1}_s(\rd))$
is Lipschitz continuous.
\end{theorem}

The main features of the
proof are given by Fourier
multiplier estimates on
\modsp s, obtained in the
previous section, together
with the following classical
iteration argument (see e.g.
\cite[Proposition
1.38]{tao}).
\begin{proposition}\label{AIA}
Let $\cN$ and $\cT$ be two
Banach spaces. Suppose we are
given a linear operator
$\cB:\cN\to \cT$ with the
bound
\begin{equation}\label{aia1}
\|\cB f\|_{\cT}\leq
C_0\|f\|_{\cN}
\end{equation}
for all $f\in\cN$ and some
$C_0>0$, and suppose that we
are given a nonlinear
operator $F:\cT\to\cN$ with
$F(0)=0$, which obeys the
Lipschitz bounds
\begin{equation}\label{aia2}
\|F(u)-F(v)\|_{\cN}\leq\frac{1}{2C_0}\|u-v\|_{\cT}
\end{equation}
for all $u,v$ in the ball
$B_\mu:=\{u\in\cT:
\|u\|_\cT\leq \mu \}$, for
some $\mu>0$. Then, for all
$u_{\rm lin}\in B_{\mu/2}$
there exists a unique
solution $u\in B_\mu$ to the
equation
\[
u=u_{\rm lin}+\cB F(u),
\]
with the map $u_{lin}\mapsto
u$ Lipschitz with constant at
most $2$ (in particular,
$\|u\|_{\cT}\leq 2\|u_{\rm
lin}\|_{\cT}$).
\end{proposition}

\begin{proof}[Proof of Theorem \ref{T1}]
We first observe that, by
Corollary \ref{C1}, for every
$0<p\leq\infty$ the
multiplier $K'(t)$, with
symbol $\cos(2\pi|\xi|)$, can
be extended to a bounded
operator on
$\mathcal{M}^{p,1}_s$, with
\begin{equation}\label{G1} \|K'(t)
u_0\|_{\mathcal{M}^{p,1}_s}\leq
C\|u_0\|_{\mathcal{M}^{p,1}_s},\quad
t\in[0,1].
\end{equation}
The uniformity of the
constant $C$, when $t$ varies
in bounded subsets, follows
from the proof of the
boundedness property itself.

Similarly, the multiplier
operator $K(t)$
 with symbol $\frac{\sin (2\pi t|\o|)}{2\pi|\o|}$,
   satisfies the
  estimate
\begin{equation}\label{G2}
\|K(t)
u_1\|_{\mathcal{M}^{p,1}_s}\leq
C
\|u_1\|_{\mathcal{M}^{p,1}_{s-1}},\
t\in[0,1],\end{equation} for
every $0<p\leq \infty$.\par
 Now we
are going to apply
Proposition \ref{AIA} with
$\cT=\cN=C^0([0,T];M^{p,1}_s)$,
where $T\leq1$ will be chosen
later on, with the nonlinear
operator $\cB$ given by the
Duhamel operator in
\eqref{op2}. Here $u_{\rm
lin}:=K'(t)u_0+K(t)u_1$ is in
the ball
$B_{\mu/2}\subset\cT$ by
\eqref{G1}, \eqref{G2}, if
$\mu$ is sufficiently large,
depending on $R$. We see that
\eqref{aia1} follows, with a
constant $C_0=O(T)$ from the
Minkowski integral inequality
(now $p\geq1$) and
\eqref{G1}.\par In order to
verify \eqref{aia2} observe
that
\begin{align*}
F(z)-F(w)&=\int_{0}^1\frac{d}{dt}
F(tz+(1-t)w)\,dt\\
&=(z-w)\sum_{j,k,l,m\geq0}c_{j,k,l,m}z^j\overline{z}^k
w^l\overline{w}^m
+(\overline{z}-\overline{w})
\sum_{j,k,l,m\geq0}c'_{j,k,l,m}z^j\overline{z}^k
w^l\overline{w}^m.
\end{align*}
Hence, applying the relation
\eqref{norms} for $q=r=1$,
 we
obtain, for $u,v\in
\cM^{p,1}_s$,
$$\|F(u)-F(v)\|_{\cM^{p,1}_s}\leq
 \|u-v\|_{\cM^{p,1}_s}\sum_{j,k,l,m\geq0}(|c_{j,k,l,m}|+
 |c'_{j,k,l,m}|)
 \|u\|_{\mathcal{\cM}^{p,1}_{s}}^{j+k}
\|v\|_{\mathcal{\cM}^{p,1}_{s}}^{l+m}
<\infty.$$
 This expression is  $\leq C_\mu \|u-v\|_{\cM^{p,1}_s}$
 if $u,v\in B_\mu$.
Hence, by choosing $T$
sufficiently small we
conclude the proof of
existence, and also that of
uniqueness among the solution
in $\cT$ with norm $O(R)$.
This last constraint can be
eliminated by a standard
continuity argument (cf. the
proof of Proposition 3.8 in
\cite{tao}).
\end{proof}

Consider now the nonlinearity
\begin{equation} \label{PW}
F(u)=F_k(u)=\lambda
|u|^{2k}u=\lambda
u^{k+1}\bar{u}^k, \quad
\lambda\in\mathbb{C},\
k\in\N.
\end{equation}
We have the following result.
\begin{theorem}\label{T2}
Let $F(u)$ as in \eqref{PW},
 $1\leq p\leq \infty$, $s\geq0$, and
 \begin{equation}\label{indr}
                q'>2kd.
\end{equation}
For every $R$ there exists
$T>0$ such that for every
$(u_0,u_1)$ in the ball $B_R$
of center $0$ and radius $R$
in $\mathcal{M}^{p,q}_s(\rd)
\times\mathcal{M}^{p,q}_{s-1}(\rd)$
there exists a unique
solution $u\in
\cC^0([0,T];\mathcal{M}^{p,q}_s(\rd))$
to \eqref{solop}. Furthermore
the map $(u_0,u_1)\mapsto u$
from $B_R$ to $
\cC^0([0,T];\mathcal{M}^{p,q}_s(\rd))$
is Lipschitz continuous.
\end{theorem}
\begin{proof} The proof goes
as the one of Theorem
\ref{T1}, but with $\cT=
\cC^0([0,T];\mathcal{M}^{p,q}_s(\rd))$,
$\cN=
\cC^0([0,T];\mathcal{M}^{p,q}_{s-1}(\rd))$.
 We only observe that now
$F_k(z)-F_k(w)=(z-w)p_k(z,w)+(\overline{z}-\overline{w})
q_k(z,w)$, where $p_k,q_k$
are polynomials of degree
$2k$ in
$z,w,\overline{z},\overline{w}$
($q_0(z,w)\equiv0$). Using
\eqref{norms} for
$0<p\leq\infty$,
 we obtain
$$
\|F(u)-F(v)\|_{\mathcal{M}^{p,r}_{s-1}}\leq
C
|\lambda|\|u-v\|_{\mathcal{M}^{p,q}_{s-1}}\|
(\|u\|_{\mathcal{M}^{p,q}_{s-1}}^{2k}+
\|v\|_{\mathcal{M}^{p,q}_{s-1}}^{2k}),$$
with
$$
r=\frac q{2k(1-q)+1}.
$$
\noindent Now, the inclusion
relations for modulation
spaces
 \cite{F1,baoxiang2} fulfill
$$\mathcal{M}^{p,r}_s\hookrightarrow
\mathcal{M}^{p,q}_{s-1}
\quad\mbox{if}\quad
\frac{d}{q}-\frac{d}{r}<1,
$$
that is \eqref{indr}. Then
\eqref{aia2} is verified and
this concludes the proof.
\end{proof}
\begin{theorem}\label{T3}
Let $F(u)$ be as in
\eqref{PW},
 $1\leq p\leq \infty$, $s\geq0$, $\gamma\geq0$, and $
                q'>2kd$.
For every $R$ there exists
$T>0$ such that for every
$(u_0,u_1)$ in the ball $B_R$
of center $0$ and radius $R$
in $W(\Fur L^q_s,L^p_\gamma)
\times W(\Fur
L^p_{s-1},L^q_\gamma)$ there
exists a unique solution
$u\in \cC^0([0,T];W(\Fur
L^q_s,L^p_\gamma))$ to
\eqref{solop}. Furthermore
the map $(u_0,u_1)\mapsto u$
from $B_R$ to $
\cC^0([0,T];W(\Fur
L^q_s,L^p_\gamma))$ is
Lipschitz continuous.\par The
same is true for an entire
real-analytic nonlinearity
$F$ as in Theorem \ref{T1}
if, in addition, $q=1$.
\end{theorem}
\begin{proof}

It follows from Corollary
\ref{C3} that the following
estimates hold:
\[
\|K'(t)u\|_{W(\Fur
L^q_s,L^p_\gamma)}\lesssim
\|u\|_{W(\Fur
L^q_{s-1},L^p_\gamma)},\quad
t\in[0,1],
\]
and
\[
\|K(t)u\|_{W(\Fur
L^q_s,L^p_\gamma)}\lesssim
\|u\|_{W(\Fur
L^q_{s},L^p_\gamma)}, \quad
t\in[0,1].
\]
We then  argue as in the
proof of the previous
results. For example, for the
nonlinearity \eqref{PW} we
choose $\cT=
\cC^0([0,T];W(\Fur
L^q_s,L^p_\gamma))$, $\cN=
\cC^0([0,T];W(\Fur
L^q_{s-1},L^p_\gamma))$. The
estimate \eqref{aia2} here
follows from the multilinear
estimate \eqref{agt} (with
$N=2k+1$) combined with the
inclusion in \eqref{incl}.
The numerology is the same as
that in the proof of Theorem
\ref{T2}. We omit the detail.
\end{proof}
\begin{remark}\rm
It would be interesting to
know whether the previous
results extend to the
(smaller) spaces
$\cM^{p,q}_s$ or $W(\Fur
L^q_s,L^p_\gamma)$, with
$p<1$. The above method of
proof does not cover this
case because of the lack of
the Minkowski integral
inequality for quasi-Banach
spaces. Indeed, for any
quasi-Banach space
$\mathcal{Q}$ there always
exist an equivalent
quasi-norm $\|\cdot\|$ and a
number $0<r<1$ so that
$\|\cdot\|^r$ satisfies the
triangle inequality. However,
the corresponding integral
version, namely
\[
\|\int_0^T u(t)\,dt\|^r\leq
\int_0^T\|u(t)\|^r dt,
\]
is false. One can see this by
taking $u(t)=av(t)$, where
$a\in \mathcal{Q}$,
$a\not=0$, is fixed, and
$v_\epsilon\in \cC([0,T];\R)$
satisfies $0\leq
v_\epsilon\leq1$,
$v_\epsilon(t)=1$ per $0\leq
t\leq\epsilon$,
$v_\epsilon(t)=0$ per $t\geq
2\epsilon$. Letting
$\epsilon\to0^+$ gives a
contradiction.

\end{remark}

\section*{Acknowledgements}
The authors would like to
thank Professor Luigi Rodino for
fruitful conversations and
comments.

\end{document}